\documentclass[10pt,psamsfonts,reqno]{amsart}

\usepackage{amssymb,latexsym}
\usepackage{amsmath,amsfonts,amsthm}
\usepackage[all]{xy}
\usepackage{graphicx}

\usepackage[left=2.25cm,top=2.5cm,bottom=2.5cm,right=2.25cm,letterpaper]{geometry}

\theoremstyle{plain}
\newtheorem{theorem}{Theorem}[section]
\newtheorem*{main}{Main Theorem}
\newtheorem{proposition}[theorem]{Proposition}
\newtheorem{lemma}[theorem]{Lemma}
\newtheorem{corollary}[theorem]{Corollary}

\theoremstyle{definition}
\newtheorem{definition}[theorem]{Definition}
\newtheorem{construction}[theorem]{Construction}
\newtheorem{remark}[theorem]{Remark}

\numberwithin{equation}{section} %equations numbered within each section

\newcommand{\N}{\mathbb N}
\newcommand{\Z}{\mathbb Z}

\newcommand{\mcC}{\mathfrak{C}}
\newcommand{\mc}[1]{\mathcal{#1}}
\newcommand{\mcA}{\mc{A}}

\newcommand{\id}[1]{\mathfrak{#1}}
\newcommand{\f}{\id{f}}
\newcommand{\cond}{\id{c}}
\newcommand{\m}{\id{m}}
\newcommand{\n}{\id{n}}

\renewcommand{\ker}{\mathop{\mathrm{Ker}}\nolimits}

\newcommand{\into}{\hookrightarrow}

\renewcommand{\bar}{\protect\overline}
\renewcommand{\hat}{\protect\widehat}
\renewcommand{\tilde}{\widetilde}

\newcommand{\rank}{\mathop{\mathrm{rank}}\nolimits}
\DeclareMathOperator{\minSpec}{MinSpec}

\newcommand{\tensor}{\otimes}
\newcommand{\End}{\mathop{\mathrm{End}}\nolimits}
\newcommand{\chr}{\ensuremath{\operatorname{char}}}

\newcommand{\ds}{\displaystyle}

\begin{document}

\title[Ranks of indecomposable modules]{Ranks of indecomposable modules over\\ rings of
infinite Cohen-Macaulay type}

\author{Andrew Crabbe}
\address{Department of Mathematics\\ Syracuse University\\ Syracuse, NY 13244-1150} \email{amcrabbe@syr.edu}

\author{Silvia Saccon}
\address{Department of Mathematics\\ University of Nebraska--Lincoln\\ Lincoln, NE 68588-0130}
\email{s-ssaccon1@math.unl.edu}

\thanks{Parts of this work appear in Saccon's Ph.D. dissertation at the University of Nebraska--Lincoln, under the supervision of Roger Wiegand.}


\date{December 20, 2011}


\begin{abstract}
Let $(R,\m,k)$ be a one-dimensional analytically unramified local ring with minimal prime ideals $P_1$, \ldots, $P_s$. Our ultimate goal is to study the direct-sum behavior of maximal Cohen-Macaulay modules over $R$. Such behavior is encoded by the monoid~$\mcC(R)$ of isomorphism classes of maximal Cohen-Macaulay $R$-modules: the structure of this monoid reveals, for example, whether or not every maximal Cohen-Macaulay module is uniquely a direct sum of indecomposable modules; when uniqueness does not hold, invariants of this monoid give a measure of how badly this property fails. The key to understanding the monoid $\mcC(R)$ is determining the ranks of indecomposable maximal Cohen-Macaulay modules. Our main technical result shows that if $R/P_1$ has infinite Cohen-Macaulay type and the residue field $k$ is infinite, then there exist $|k|$ pairwise non-isomorphic indecomposable maximal Cohen-Macaulay $R$-modules of rank $(r_1,\dots,r_s)$ provided $r_1 \ge r_i$ for all $i\in \{1,\ldots, s\}$. This result allows us to describe the monoid $\mcC(R)$ when $\hat{R}/Q$ has infinite Cohen-Macaulay type for every minimal prime ideal $Q$ of the $\m$-adic completion $\hat{R}$. 
\end{abstract}

\maketitle

\section{Introduction}
A commutative ring $R$ is \textit{local} if $R$ is a Noetherian ring with exactly one maximal ideal, and a local ring $(R,\n)$ is \textit{analytically unramified} if the $\n$-adic completion $\hat{R}$ of $R$ is reduced. If $R$ is complete, then the Krull-Remak-Schmidt property holds for the class of finitely generated $R$-modules; that is, direct-sum decompositions of finitely generated $R$-modules are unique (see~\cite[Theorem~5.20]{S70}). However, there are examples of non-complete local rings for which direct-sum decomposition is not unique (see for example~\cite[Section~1]{E73} or~\cite[Sections~1 and~2]{rW99}). Thus it is natural to ask how modules behave over $R$, especially when $R$ is a non-complete local ring.

In this paper, we restrict our attention to one-dimensional analytically unramified local rings $(R,\m,k)$, and to the class of maximal Cohen-Macaulay $R$-modules. In this context, a maximal Cohen-Macaulay $R$-module is exactly a non-zero finitely generated torsion-free $R$-module. We say $R$ has \textit{finite Cohen-Macaulay type} if there exist only finitely many indecomposable maximal Cohen-Macaulay $R$-modules, up to isomorphism. Otherwise, we say $R$ has \textit{infinite Cohen-Macaulay type}.

One approach to the study of direct-sum behavior of modules over $R$ is to describe the monoid $\mcC(R)$ of isomorphism classes of maximal Cohen-Macaulay $R$-modules (together with $[0]$), with operation given by $[M]+[N]=[M\oplus N]$. This monoid is free if and only if the Krull-Remak-Schmidt property holds for the class of maximal Cohen-Macaulay $R$-modules; moreover, some invariants of this monoid give a measure of how badly the Krull-Remak-Schmidt property can fail. 

The notion of rank of a module plays a fundamental role in studying the monoid $\mcC(R)$, and thus in understanding the direct-sum behavior of modules over $R$. The \textit{rank} of an $R$-module~$M$ is the tuple consisting of the vector-space dimensions of $M_P$ over $R_P$, where $P$ ranges over the minimal prime ideals of $R$. More precisely, let $P_1$, \ldots, $P_s$ denote the minimal prime ideals of $R$. The rank of~$M$ at the minimal prime ideal $P_i$ is $r_i=\dim_{R_{P_i}}M_{P_i}$, the dimension of the vector space $M_{P_i}$ over the field $R_{P_i}$. The rank of $M$ is the $s$-tuple $(r_1,\ldots, r_s)$. In this paper,  we are primarily interested in the following question: What ranks occur for indecomposable maximal Cohen-Macaulay $R$-modules? 
By answering this question, we will be able to give a precise description of the monoid $\mcC(R)$. %and hence to understand the direct-sum behavior of maximal Cohen-Macaulay $R$-modules. 
%Answering this question will allow us to describe the monoid~$\mcC(R)$ of isomorphism classes  of maximal Cohen-Macaulay $R$-modules (together with $[0]$).% when $\hat{R}/Q$ has infinite Cohen-Macaulay type for all minimal prime ideals $Q$ of $\hat{R}$.

%In this paper, our goal is to describe the monoid $\mcC(R)$ of isomorphism classes  of maximal Cohen-Macaulay $R$-modules when $\hat{R}/Q$ has infinite Cohen-Macaulay type for all minimal prime ideals $Q$ of $\hat{R}$. Thus we are primarily interested in the following question: What ranks occur for indecomposable maximal Cohen-Macaulay $R$-modules?

When $R$ has finite Cohen-Macaulay type, N.~Baeth and M.~Luckas give a complete answer to this question in~\cite[Main Theorem]{BL09}. Baeth and Luckas classify all possible ranks of indecomposable maximal Cohen-Macaulay $R$-modules depending on the number of minimal prime ideals (over such rings there are at most three minimal prime ideals). Their work builds upon results of Green and Reiner~\cite[pages~76-77, 81-82]{GR78}, R.~Wiegand and
S.~Wiegand~\cite[Theorem~3.9]{rWsW94}, \c{C}imen~\cite[Theorem~2.2]{C94}, Baeth~\cite[Theorem~4.2]{B07}, and Arnavut, Luckas and S.~Wiegand~\cite[Theorem~2.3]{ALsW07}. Moreover, in~\cite[Theorem~3.4]{BLunp}, Baeth and Luckas use the list of ranks of indecomposable maximal Cohen-Macaulay $R$-modules to give a description of the monoid $\mcC(R)$.% of isomorphism classes of maximal Cohen-Macaulay $R$-modules (together with $[0]$).

The question is open when $R$ has infinite Cohen-Macaulay type. In~\cite[Section~2]{rW89}, R.~Wiegand constructs an indecomposable maximal Cohen-Macaulay module of constant rank $(r,\ldots, r)$ for each positive integer $r$. In~\cite[Lemma~2.2]{rW01}, he shows that if $R/P_1$ satisfies a certain specific condition (corresponding to Case (1) of Remark~\ref{R:remark1}), and if $r_1>0$ and $r_1\ge r_i$ for all $i\in \{1,\ldots, s\}$, then there is an indecomposable maximal
Cohen-Macaulay $R$-module of rank $(r_1, \ldots, r_s)$.

In this article, we prove a stronger result: we generalize Wiegand's theorem and we determine the cardinality of the set of isomorphism classes of indecomposable maximal Cohen-Macaulay $R$-modules. Our goal is to prove the following theorem.

\begin{main}
Let $(R,\m,k)$ be a one-dimensional analytically unramified local ring with minimal prime ideals $P_1$, \ldots, $P_s$. Assume $R/P_1$ has infinite Cohen-Macaulay type. Let $(r_1, \ldots, r_s)$ be a non-trivial $s$-tuple of non-negative integers with $r_1\ge r_i$ for all $i\in \{1,\ldots, s\}$.
\begin{enumerate}
\item There exists an indecomposable maximal Cohen-Macaulay $R$-module of rank $(r_1,
\ldots, r_s)$.
\item If the residue field $k$ is infinite, then the set of isomorphism classes of indecomposable maximal Cohen-Macaulay $R$-modules of rank $(r_1, \ldots, r_s)$ has cardinality $|k|$.
\end{enumerate}
\end{main}

In the proof of this result, we put a specific construction due to R.~Wiegand in a general context. This introduces various technical difficulties, particularly in the cases when the multiplicity of the ring is three, or when the residue field is not of characteristic zero. There are two key ingredients in the proof, namely:
\begin{enumerate}
\item [(a)] Proposition~\ref{P:prop} that allows us to focus our attention to finite dimensional $k$-algebras of dimension at least four, and 
\item [(b)] the construction of a family of modules of a certain rank---we show that these modules are indecompo\-sa\-ble in all the cases that can arise for a finite-dimensional $k$-algebra of dimension at least four. This new construction (depending on a parameter $t\in k$) allows us to determine the cardinality of the set of isomorphism classes of indecomposable maximal Cohen-Macaulay modules. 
\end{enumerate}

In Section~\ref{S:background}, we provide some background and some known results. In Section~\ref{S:main}, we prove the main theorem and use this result to describe the monoid $\mcC(R)$ when $\hat{R}/Q$ has infinite Cohen-Macaulay type for all minimal prime ideals $Q$ of $\hat{R}$. In Section~\ref{S:example}, we study a particular ring $R$ with exactly two minimal prime ideals $P_1$ and $P_2$ such that $R/P_1$ has infinite Cohen-Macaulay type and $R/P_2$ has finite Cohen-Macaulay type; we show that certain ranks cannot occur for this example.

\section{Background} \label{S:background}

One-dimensional reduced local rings $(R,\m)$ of finite Cohen-Macaulay type are completely cha\-racte\-rized as those rings satisfying the following two conditions of Drozd and Ro{\u\i}ter (see~\cite{DR67}):
\begin{list}{}{\setlength{\leftmargin}{1.7cm}}
\item[(\textbf{DR1})] the integral closure $\overline{R}$ of $R$ in the total quotient ring $Q(R)$ of $R$ is generated by at most three elements as an $R$-module;
\item[(\textbf{DR2})] $(\m\overline{R}+R)/R$ is a cyclic $R$-module.
\end{list}
These two conditions were first introduced in~\cite{DR67} by Drozd and Ro{\u\i}ter, who proved the theorem in the special case of a ring which is a localization of a module-finite $\Z$-algebra.
R.~Wiegand proved necessity of the two conditions in~\cite{rW89} and sufficiency, except in the case when $k$ is imperfect of characteristic~$2$, in~\cite{rW89} and~\cite{rW94}. \c{C}imen proved the remaining case of the theorem in his Ph.D. dissertation~\cite{C94}.

The main tool for constructing indecomposable (maximal Cohen-Macaulay) modules is the Artinian pair (see for example~\cite[Section 1]{rW89}).

\begin{definition}
An \textit{Artinian pair} $A\into B$ is a module-finite extension of commutative Artinian rings.
\end{definition}

\begin{definition}
A \textit{module} over the Artinian pair $A\into B$ is a pair $V\into W$ such that
\begin{enumerate}
\item $W$ is a finitely generated projective $B$-module,
\item $V$ is an $A$-submodule of $W$ and
\item $BV=W$.
\end{enumerate}
\end{definition}

\begin{definition}
A \textit{homomorphism} from an $(A\into B)$-module $V_1\into W_1$ to an $(A\into B)$-module $V_2\into W_2$ is a $B$-module homomorphism $\varphi\colon W_1\to W_2$ such that
$\varphi(V_1)\subseteq V_2$.
\end{definition}

Thus there are notions of isomorphism of $(A\into B)$-modules, endomorphism rings $\End_B(V,W)$ of an $(A\into B)$-module $V\into W$ and direct-sum decompositions of $V\into W$.

\begin{definition}
We say that an Artinian pair $A\into B$ has \textit{finite representation type} if there exist only finitely many indecomposable $(A\into B)$-modules, up to isomorphism.
\end{definition}

Given a one-dimensional analytically unramified local ring $(R,\m,k)$ with integral closure $\overline{R}\ne R$, we can associate to $R$ an Artinian pair $R_\text{art}$ in the following
way. Let $\cond=\{r\in R \mid r\overline{R}\subseteq R\}$ denote the conductor of $R$, i.e., the largest ideal of $\overline{R}$ contained in $R$. As $\cond$ contains a non-zero-divisor of $R$,
both $R/\cond$ and $\overline{R}/\cond$ are Artinian rings, and $R/\cond\into \overline{R}/\cond$ is an Artinian pair, denoted~$R_{\text{art}}$. We have a pullback diagram, called the
\textit{conductor square} of $R$:
\[
\xymatrix{R\ar@{^{(}->}[r] \ar@{->>}[d] & \overline{R} \ar@{->>}[d]\\
\ds \frac{R}{\cond}\ar@{^{(}->}[r] & \ds \frac{\overline{R}}{\cond}
.}
\]
Given a maximal Cohen-Macaulay $R$-module $M$, we can associate to $M$ an $R_\text{art}$-module $M_\text{art}=(M/\cond M \into \overline{R}M/\cond M)$, where $\overline{R}M$ denotes the $\overline{R}$-submodule of $Q(R)\tensor_R M$ generated by the image of~$M$.

For a one-dimensional analytically unramified local ring $(R,\m,k)$, the Drozd and Ro{\u\i}ter conditions can be interpreted in terms of the Artinian pair $R_\text{art}=(A\into B)$, where $(A,\n,k)$ is a local ring. In this case, the Drozd and Ro{\u\i}ter conditions  (\textbf{DR1}) and (\textbf{DR2}) are equivalent to:
\begin{list}{}{\setlength{\leftmargin}{1.7cm}}
\item[(\textbf{dr1})] $\ds \dim_k \frac{B}{\n B}\le 3$,
\item[(\textbf{dr2})] $\ds \dim_k\left(\frac{\n B+A}{\n^2B+A}\right)\le 1$.
\end{list}
That is, $R$ satisfies (\textbf{DR1}) and (\textbf{DR2}) if and only if $R_\text{art}$ satisfies (\textbf{dr1}) and (\textbf{dr2}).

In order to study maximal Cohen-Macaulay $R$-modules, it is convenient to work in the category of $R_\text{art}$-modules. The following theorems summarize some facts about the relationships
between $R$ and $R_\text{art}$, and the modules $M$ and $M_\text{art}$.

\begin{theorem}[\cite{rW89}, Propositions 1.5-1.9] \label{T:theorem1}
Let $(R,\m,k)$ be a one-dimensional analytically unramified local ring, and let $R_{\normalfont\text{art}}$ denote the corresponding Artinian pair. Let $\overline{R}$ denote the integral closure of $R$ in its total quotient ring, and assume $\overline{R}\ne R$. Let $M$, $N$ be maximal Cohen-Macaulay $R$-modules, and let $V\into W$ be an $R_{\normalfont\text{art}}$-module.
\begin{enumerate}
\item $(V\into W)\cong X_{\normalfont\text{art}}$ for some maximal Cohen-Macaulay
$R$-module $X$ if and only if $W\cong F/\cond F$ for some finitely generated projective $\overline{R}$-module $F$.
\item $(M\oplus N)_{\normalfont\text{art}}\cong M_{\normalfont\text{art}}\oplus N_{\normalfont\text{art}}$.
\item $M_{\normalfont\text{art}}\cong N_{\normalfont\text{art}}$ if and only if $M\cong N$.
\item The Krull-Remak-Schmidt theorem holds for direct-sum decompositions
of modules over $R_{\normalfont\text{art}}$.
\item $R$ has finite Cohen-Macaulay type if and only if $R_{\normalfont\text{art}}$ has finite representation type.
\end{enumerate}
\end{theorem}

\begin{theorem}[\cite{rWsW94}, Proposition 2.2] \label{T:theorem2}
With the notation as in Theorem \ref{T:theorem1}, assume $\overline{R}$ is a direct product of local rings.
\begin{enumerate}
\item The functor $M\mapsto M_{\normalfont\text{art}}$ determines a bijection between the set of isomorphism classes of maximal Cohen-Macaulay $R$-modules and the set of isomorphism classes of $R_{\normalfont\text{art}}$-modules.
\item $M$ is indecomposable if and only if $M_{\normalfont\text{art}}$ is indecomposable.
\end{enumerate}
\end{theorem}

\begin{proposition}\label{P:local}
Let $(R,\m,k)$ be a one-dimensional analytically unramified local ring with minimal prime ideals $P_1$, \ldots, $P_s$. Let $\overline{R}$ denote the integral closure of $R$ in its total quotient ring, and assume $\overline{R}\ne R$. Set $R_i=R/P_i$ and $B_i=\overline{R}_i/\cond
\overline{R}_i$ for $i\in \{1,\ldots, s\}$. If $V\into W$ is an indecomposable $R_{\normalfont\text{art}}$-module with $W=B_1^{(r_1)}\times\cdots\times B_s^{(r_s)}$, then $(V\into W) \cong X_{\normalfont\text{art}}$ for an indecomposable maximal Cohen-Macaulay $R$-module $X$ of rank $(r_1,\ldots, r_s)$.
\end{proposition}

\begin{proof}
Let $F= \overline{R}_1^{(r_1)}\times \cdots\times \overline{R}_s^{(r_s)}$ be a projective $\overline{R}$-module. Since $W \cong F/\cond F$, Theorem~\ref{T:theorem1}(2) guarantees the existence of a maximal Cohen-Macaulay $R$-module $X$ such that $X_{\text{art}} \cong (V\into W)$.

Suppose $X=X_1\oplus X_2$ for $R$-modules $X_1$ and $X_2$. Then $X_{\text{art}}\cong X_{1,\text{art}}\oplus X_{2,\text{art}}$, and since $X_{\text{art}}$ is indecomposable, $X_{1,\text{art}}=0$ or $X_{2,\text{art}}=0$. Without loss of generality, assume $X_{1,\text{art}}=(X_1/\cond X_1\into \overline{R}X_1/\cond X_1)=0$; in particular, the first component $X_1/\cond X_1$ of
$X_{1,\text{art}}$ is zero. Since $\cond\subseteq \m$, by Nakayama's lemma we conclude that $X_1=0$. Thus $X$ is indecomposable. Since $\overline{R}X \cong F$, we see that $\rank (X)=(r_1, \ldots, r_s)$.
\end{proof}

It is often useful to study the Artinian pair $k\into D$ obtained from the Artinian pair $A\into B$ by reducing modulo the maximal ideal of $A$. We use the following results in~\cite{rW89}.

\begin{lemma}[\cite{rW89}, Lemma 2.4]\label{L:lemma2}
Let $A\into B$ be an Artinian pair, and let $I$ be a nilpotent ideal of $B$. The functor from the category of $(A\into B)$-modules to the category of $(A/(I\cap A)\into B/I)$-modules defined by
\[
(V\into W)\mapsto \left(\frac{V+IW}{IW}\into \frac{W}{IW}\right)
\]
is surjective on isomorphism classes and reflects indecomposable objects.
\end{lemma}

\begin{lemma}[\cite{rW89}, Lemma 2.4]\label{L:lemma3}
Let $A\into B$ be an Artinian pair, and let $C$ be a ring between $A$ and $B$. The functor from the category of $(A\into C)$-modules to the category of $(A\into B)$-modules defined by
\[
(V\into W)\mapsto (V\into B\tensor_C W)
\]
is full and faithful, and preserves and reflects indecomposable objects.
\end{lemma}

\section{Indecomposable maximal Cohen-Macaulay modules of a given rank} \label{S:main}

It is useful to reduce the construction of indecomposable modules over $R$ to the construction of indecomposable modules over the Artinian pair $R_\text{art}$. The following proposition allows one to replace the Artinian pair $R_\text{art}=(A\into B)$ with the Artinian pair $k\into D$, where $k$ is a field and $D$ is a finite-dimensional $k$-algebra satisfying certain conditions.

\begin{proposition}\label{P:prop}
Let $(R,\m,k)$ be a one-dimensional analytically unramified local ring with minimal prime ideals $P_1$, \ldots, $P_s$. Assume $R/P_1$ has infinite Cohen-Macaulay type. Let $R_{\normalfont\text{art}}=(A\into B)$ be the corresponding Artinian pair. Then there is a ring $C$ between $A$ and $B$ such that, reducing modulo the maximal ideal of $A$, we have
\[
\xymatrix{A\ar@{^{(}->}[r] \ar@{->>}[d]& C\ar@{->>}[d]\\
k\ar@{^{(}->}[r] & D,}
\]
where $D=D_1\times \cdots \times D_s$ is a product of finite-dimensional $k$-algebras and either
\[
\dim_k D_1\ge 4 \quad \textrm{ or } \quad D_1\cong
k[X,Y]/(X^2,XY,Y^2).
\]
Moreover, if there exists an indecomposable $(k\into D)$-module $V\into W$, where $W=D_1^{(r_1)}\times \cdots \times D_s^{(r_s)}$, then there exists an indecomposable maximal Cohen-Macaulay $R$-module of rank $(r_1,\ldots, r_s)$.
\end{proposition}

\begin{proof}
Let $\cond$ be the conductor of $R$, and let $R_i$ denote the integral domain $R/P_i$ for $i\in \{1,\ldots,s\}$. Consider the conductor square of $R$:
\[
\xymatrix{ R \ar[d] \ar@{^{(}->}[r] & \ds R_1\times \cdots \times
R_s\ar@{^{(}->}[r] & \overline{R}_1\times \cdots \times
\overline{R}_s=\overline{R}\ar[d]
\\ **[l] A=\ds\frac{R}{\cond}
\ar[r] \ar@{^{(}->}[r] & \ds \frac{R_1}{\cond R_1}\times \cdots
\times \frac{R_s}{\cond R_s} \ar@{^{(}->}[r] & \ds B_1\times \cdots
\times B_s=\frac{\overline{R}}{\cond}=B, }
\]
where $B_i=\overline{R}_i/\cond \overline{R}_i$, $i\in\{1,\ldots,s\}$. Now consider the conductor square for $R_1=R/P_1$, and let $\cond_1$ denote the conductor of $R_1$ (note that $\cond_1\supseteq \cond\overline{R}_1=\cond R_1$):
\[
\xymatrix{ R_1 \ar[d] \ar@{^{(}->}[r] & \overline{R}_1\ar[d]
\\ **[l] A_1=\ds\frac{R_1}{\cond \overline{R}_1}
\ar[r] \ar@{^{(}->}[r] \ar[d] & **[r]
\displaystyle{\frac{\overline{R}_1}{\cond \overline{R}_1}}=B_1 \ar[d]\\
**[l] \tilde{A}_1=\ds \frac{R_1}{\cond_1} \ar[r] \ar@{^{(}->}[r] &
**[r] \displaystyle{\frac{\overline{R}_1}{\cond_1}}=\tilde{B}_1. }
\]
Since $\overline{R}$ is semilocal and integrally closed, $\overline{R}$ is a principal ideal ring. Hence $\overline{R}_1$ is a principal ideal domain, and so $\tilde{B}_1$ is a principal ideal
ring. Also, note that $\tilde{A}_1$ is local with maximal ideal $\n$. Since $R_1$ has infinite Cohen-Macaulay type, the Drozd-Ro{\u\i}ter conditions fail for the corresponding Artinian
pair $\tilde{A}_1\into \tilde{B}_1$. By Proposition~2.3 in~\cite{rW89}, there is a ring $\tilde{C}_1$ between $\tilde{A}_1$ and $\tilde{B}_1$ such that, setting $\tilde{D}_1=\tilde{C}_1/\n\tilde{C}_1$, we have either
\begin{equation}\label{E:eqn10}
\dim_k \tilde{D}_1\ge 4 \quad \textrm{ or } \quad \tilde{D}_1\cong
k[X,Y]/(X^2,XY,Y^2).
\end{equation}
(If (\textbf{DR1}) fails, then $\dim_k (\tilde{B}_1/\n \tilde{B}_1)\ge 4$ and we take $\tilde{C}_1=\tilde{B}_1$; otherwise (\textbf{DR2}) fails and we take $\tilde{C}_1=\tilde{A}_1+\n\tilde{B}_1$.)

Now let $\f= \cond_1/\cond\overline{R}_1$ denote the kernel of the lower-right vertical map in the diagram above, and let $C_1$ be the ring between $A_1$ and $B_1$ such that $C_1/\f=\tilde{C}_1$. Finally, set $D_1=C_1/\n C_1$, and observe that $D_1$ and $\tilde{D}_1$ are naturally isomorphic.

Then $C=C_1\times B_2\times \cdots\times B_s$ is a ring between $A$ and $B$. Reducing modulo the maximal ideal~$\n$ of $A$, we have
\[
\xymatrix{ A \ar@{^{(}->}[r]\ar@{->>}[d] & C \ar@{->>}[d]\\
k\ar@{^{(}->}[r] & D,}
\]
where $D=D_1\times D_2\times \cdots \times D_s$, $D_1=C_1/\n C_1$ and $D_i=B_i/\n B_i$ for $i\in\{2,\ldots,s\}$. In particular, by~\eqref{E:eqn10}, we have either $\dim_k D_1\ge 4$ or $D_1\cong k[X,Y]/(X^2,XY,Y^2)$.

By Lemma~\ref{L:lemma2}, if there exists an indecomposable $(k\into D)$-module $V_1\into W_1$ with $W_1=D_1^{(r_1)}\times \cdots \times D_s^{(r_s)}$, then there exists an indecomposable $(A\into C)$-module $V_2 \into W_2$ with $W_2= C_1^{(r_1)}\times
B_2^{(r_2)}\times\cdots \times B_s^{(r_s)}$. Hence, by Lemma~\ref{L:lemma3}, there exists an indecomposable $(A\into B)$-module $V\into W$ with $W=B_1^{(r_1)}\times \cdots \times
B_s^{(r_s)}$. By Proposition~\ref{P:local}, there exists an indecomposable maximal Cohen-Macaulay $R$-module of rank $(r_1, \ldots, r_s)$.
\end{proof}

The following remark will be useful in the proof of the main theorem.

\begin{remark}\label{R:remark1}
Let $k$ be a field, and let $D$ be a finite-dimensional $k$-algebra with $\dim_k D\ge 4$. The following are all the cases that can arise (see the proof of~\cite[Proposition~2.6]{rW89} and~\cite[pages 5-7]{KrW}).
\begin{enumerate}
\item There exists an element $a\in D$ such that $\{1,a,a^2\}$ is linearly independent over
$k$.
\item For all $a\in D$, $\{1,a,a^2\}$ is linearly dependent over $k$. In this case, one of the following holds.
\begin{enumerate}
\item There exist elements $a$, $b\in D$ such that $\{1,a,b\}$ is linearly independent over $k$ and $a^2=ab=b^2=0$.
\item There exist elements $a$, $b\in D$ such that $\{1,a,ab,b\}$ is linearly independent over $k$ and $a^2=b^2=0$.
\item The characteristic of $k$ is $2$, $(D,\m_D,K_D)$ is a local ring, $K_D$ is a purely inseparable extension of $k$ of degree at least four and $a^2\in k$ for all $a\in K_D$.
\item The field $k$ has cardinality $2$ and $D=D_1\times \cdots\times D_l$, where $l\ge 4$ and each $(D_i,\m_i)$ is an Artinian local ring such that $D_i/\m_i=k$.
\end{enumerate}
\end{enumerate}
Observe that in case (1), there are elements $a$, $b\in D$ such that $\{1,a,a^2,b\}$ is linearly independent over $k$. In case (2c), there are elements $a$, $b\in K_D$ such that $\{1,a,ab,b\}$ is
linearly independent over $k$ and $a^2$, $b^2\in k$ (indeed, choose $a$, $b$ such that $[k(a,b):k]=4$).
\end{remark}

\subsection{Building modules over Artinian pairs}

Let $k$ be a field, and let $D=D_1\times \cdots\times D_s$, where each $D_i$ is a finite-dimensional $k$-algebra. Assume $\dim_k D_1\ge 3$. Let $(r_1,\ldots, r_s)$ be a non-trivial $s$-tuple of non-negative integers with $r_1\ge r_i$ for all $i\in \{1,\ldots, s\}$. We shall build explicit $(k\into D)$-modules $V\into W$, where $W=D_1^{(r_1)}\times\cdots\times D_s^{(r_s)}$. The first
construction is a generalization of constructions found in~\cite[Construction~2.5]{rW89} and~\cite[Lemma~2.2]{rW01}. The second construction is a generalization of Dade's construction (see~\cite{D63} and~\cite[Theorem~2.9]{rW89}).

\begin{construction}\label{construction1}
Let $W=D_1^{(r_1)}\times\cdots\times D_s^{(r_s)}$, and let $k^{(r_1)}$ be the $k$-vector space of $r_1\times 1$ column vectors. Let $\partial\colon k^{(r_1)}\to W$ be the ``truncated diagonal
map'', given by:
\[
\begin{bmatrix}
c_1\\ \vdots\\ c_{r_1}
\end{bmatrix} \mapsto \left(\begin{bmatrix}
c_1\\ \vdots\\ c_{r_1}
\end{bmatrix}, \begin{bmatrix}
c_1\\ \vdots\\ c_{r_2}
\end{bmatrix}, \ldots, \begin{bmatrix}
c_1\\ \vdots\\ c_{r_s}
\end{bmatrix}\right).
\]
Choose elements $a_1$, $b_1\in D_1$ such that $\{1,a_1,b_1\}$ is linearly independent over $k$. Let $a=(a_1,0,\ldots, 0)$ and $b=(b_1, 0,\ldots, 0)$ be elements in $D$. For $i\in \{1,\ldots, s\}$, let $H_i$ be the $r_i\times r_i$ nilpotent matrix with $1$ on the superdiagonal and $0$ elsewhere.

For $t\in k$, let $V_t$ be the $k$-subspace of $W$ consisting of elements of the form
\[
\partial(x)+(a+tb)\partial(y)+b\partial(H_1y),
\]
where $x$ and $y$ range over $k^{(r_1)}$. Let $e_{r_j,i}\in D_j^{(r_j)}$ denote the standard unit vector in $D_j^{(r_j)}$ with $1$ in the $i$-th position. We see that every element in $W$ is a
$D$-linear combination of elements of the form $e_W=(0,\ldots, e_{r_j,i}, 0,\ldots,0)$, where $i\in \{1,\ldots,r_j\}$ and $j\in \{1,\ldots,s\}$. Note that we can consider $e_{r_j,i}\in k^{(r_j)}$ as an element of $k^{(r_1)}$ (by appending zeros if $r_1>r_i$). Let $\varepsilon_j\in D$ be an idempotent with support $\{j\}$. Then $e_W=\varepsilon_j\partial(e_{r_j,i})\in DV_t$. Hence $V_t\into W$ is a $(k\into D)$-module with $W=D_1^{(r_1)}\times\cdots\times D_s^{(r_s)}$.
\end{construction}

\begin{construction}\label{construction2}
Assume $D_1=D_{1,1}\times D_{1,2}\times \cdots \times D_{1,l}$, where $l\ge 4$ and each $(D_{1,i},\m_{1,i})$ is an Artinian local ring such that $D_{1,i}/\m_{1,i}=k$.

Let $I=(\m_{1,1}\times \cdots \times \m_{1,l})\times (0)\times \cdots \times (0)$, and let $E=D/I=(k\times \cdots\times k)\times D_2\times \cdots \times D_s$. We construct a $(k\into E)$-module as
follows. For $i\in \{2,\ldots, s\}$, let $\partial_i\colon k^{(r_1)}\to D_i^{(r_i)}$ denote the map
\[
\begin{bmatrix}
c_1\\ \vdots\\ c_{r_1}
\end{bmatrix} \mapsto \begin{bmatrix}
c_1\\ \vdots\\ c_{r_i}
\end{bmatrix}.
\]
Let $W=(k^{(r_1)}\times \cdots\times k^{(r_1)})\times D_2^{(r_2)}\times \cdots\times D_s^{(r_s)}$. Let $V$ be the $k$-subspace of $W$ consisting of elements of the form
\[
(x,y,x+y,x+H_1y,x,\ldots,x,\partial_2(x),\ldots,\partial_s(x)),
\]
where $x$ and $y$ range over $k^{(r_1)}$. Note that $V$ contains $(x,0,x,x,x,\ldots,x,\partial_2(x),\ldots,\partial_s(x))$ for all $x\in k^{(r_1)}$, and $(0,y,y,H_1y,0,\ldots,0)$ for all $y\in k^{(r_1)}$. Using idempotents of $E$, we see that $EV=W$. Hence $V\into W$ is a $(k\into E)$-module with $W=(k^{(r_1)}\times \cdots\times k^{(r_1)})\times D_2^{(r_2)}\times \cdots\times D_s^{(r_s)}$.
\end{construction}

\subsection{Main Result}
Using the constructions described above, we now state and prove the main theorem.

\begin{main}\label{T:main}
Let $(R,\m,k)$ be a one-dimensional analytically unramified local ring with minimal prime ideals $P_1$, \ldots, $P_s$. Assume $R/P_1$ has infinite Cohen-Macaulay type. Let $(r_1, \ldots, r_s)$ be a non-trivial $s$-tuple of non-negative integers with $r_1\ge r_i$ for all $i\in \{1,\ldots, s\}$.
\begin{enumerate}
\item There exists an indecomposable maximal Cohen-Macaulay $R$-module of rank $(r_1,
\ldots, r_s)$.
\item If the residue field $k$ is infinite, then the set of isomorphism classes of indecomposable maximal Cohen-Macaulay $R$-modules of rank $(r_1, \ldots, r_s)$ has cardinality $|k|$.
\end{enumerate}
\end{main}

\begin{proof}
With the notation as in Proposition~\ref{P:prop}, it is enough to construct the desired modules over the Artinian pair $k\into D$, where $D=D_1\times \cdots \times D_s$ and either $\dim_k D_1\ge 4$
or $D_1\cong k[X,Y]/(X^2,XY,Y^2)$.

Let $W=D_1^{(r_1)}\times\cdots\times D_s^{(r_s)}$. Choose $a_1$, $b_1\in D_1$ such that $\{1,a_1,b_1\}$ is linearly independent over $k$, and let $V_t\into W$ be the $(k\into D)$-module given in Construction~\ref{construction1}.

Observe that, when $k$ is infinite, the set $\{1,a_1,b_1,(a_1+tb_1)^2\}$ is linearly independent over $k$ for all but finitely many $t\in k$; moreover, for a fixed $t$, the set $\{1,a_1,b_1,(a_1+tb_1)(a_1+ub_1)\}$ is linearly independent over~$k$ for all but finitely many $u\in k$. We show the following:
\begin{enumerate}
\item [(i)] there is $t\in k$ such that $V_t\into W$ is indecomposable; moreover, if $k$ is infinite, then $V_t\into W$ is indecomposable for all but finitely many $t$;
\item [(ii)] if $u\ne t$ and both $\{1,a_1,b_1,(a_1+tb_1)^2\}$ and $\{1,a_1,b_1,(a_1+tb_1)(a_1+ut_1)\}$ are linearly independent over~$k$, then $V_t\into W$ is not isomorphic to $V_u\into W$. By the previous observation, when $k$ is infinite, there are $|k|$ pairwise non-isomorphic indecomposable $(k\into D)$-modules.
\end{enumerate}

Consider a morphism $\varphi\colon(V_t\into W)\to (V_u\into W)$. Since $\varphi(V_t)\subseteq V_u$ and $\partial(x)\in V_t$ for all $x\in k^{(r_1)}$, we can write
\[
\varphi(\partial(x))=\partial(x')+(a+ub)\partial(x'')+b\partial(H_1x'')
\]
for elements $x'$, $x''\in k^{(r_1)}$. By linear independence of $\{1,a_1, b_1\}$ over $k$, the elements $x'$ and $x''$ are uniquely and linearly determined by $x$. Thus there exist $r_1\times r_1$ matrices $\sigma$ and $\tau$ with entries in $k$ such that
\begin{equation}\label{E:eqn5}
\varphi(\partial(x))=\partial(\sigma x)+(a+ub)\partial(\tau x)+b\partial(H_1\tau x)
\end{equation}
for all $x\in k^{(r_1)}$. Similarly, there are $r_1\times r_1$ matrices $\mu$ and $\rho$ with entries in $k$ such that
\begin{equation}\label{E:eqn6}
\varphi((a+tb)\partial(y)+b\partial(H_1y))=\partial(\mu y)+(a+ub)\partial(\rho y)+b\partial(H_1\rho y)
\end{equation}
for all $y\in k^{(r_1)}$. Since $\varphi$ is $D$-linear, we can rewrite~\eqref{E:eqn6} as:
\begin{align}
-\partial(\mu y)+a\partial((\sigma-\rho)y)&+b\partial((t\sigma -u\rho+\sigma H_1 - H_1\rho)y)+ (a+tb)(a+ub)\partial(\tau y)\nonumber
\\ &+ab\partial((H_1\tau+\tau H_1)y)+b^2\partial((tH_1\tau +u\tau H_1 +H_1\tau H_1)y)=0\label{E:eqn7}
\end{align}
for all $y\in k^{(r_1)}$. Moreover, if $u=t$ and $\varphi$ is an idempotent, then the equation
$\varphi^2(\partial(x))=\varphi(\partial(x))$ gives:
\begin{align}
\partial(\sigma^2x)&+a\partial((\tau\sigma+\sigma\tau) x)+b\partial((t(\tau\sigma+\sigma\tau)+H_1\tau\sigma+\sigma H_1\tau) x)+ (a+tb)b\partial((H_1\tau^2 + \tau H_1\tau )x)\nonumber\\
&+(a+tb)^2\partial(\tau^2 x)+b^2\partial(H_1\tau H_1\tau x) =\partial(\sigma x)+a\partial(\tau x)+ b\partial((tI+H_1)\tau x)
\label{E:eqn11}
\end{align}
for all $x\in k^{(r_1)}$.

First suppose $\{1,a_1,b_1\}$ is linearly independent over $k$ and $a_1^2=a_1b_1=b_1^2=0$. (Note that if $D_1\cong k[X,Y]/(X^2,XY,Y^2)$, then this hypothesis is clearly satisfied.)
Equation~\eqref{E:eqn7} simplifies to
\begin{equation}\label{E:eqn21}
-\partial(\mu y)+a\partial((\sigma-\rho)y)+b\partial((t\sigma -
u\rho+\sigma H_1 - H_1\rho)y) =0.
\end{equation}
Using linear independence of $\{1,a_1,b_1\}$ over $k$, we obtain the equations
\[
\mu=0, \qquad \sigma=\rho, \qquad t\sigma-u\rho+\sigma H_1 - H_1\rho=0,
\]
and hence
\begin{equation}\label{E:eqn8}
\sigma[(t-u)I+H_1]=H_1\sigma.
\end{equation}
If $\varphi$ is an isomorphism, then from equation~\eqref{E:eqn5} we see that $\sigma$ is invertible (use the facts that $a_1$, $b_1$ are nilpotent elements and that if $\alpha=\beta+\gamma$ for a unit $\alpha$ and a nilpotent element $\gamma$, then $\beta$ is a unit).
This is a contradiction to equation~\eqref{E:eqn8} if $u\ne t$. Hence, if $u\ne t$, then $(V_t\into W)\not\cong (V_u\into W)$.

To show that $V_t\into W$ is indecomposable, take $u=t$ and assume $\varphi$ is an idempotent. Equation~\eqref{E:eqn11} simplifies to
\[
\partial(\sigma^2x)+a\partial((\tau\sigma+\sigma\tau) x)+b\partial((t(\tau\sigma+\sigma\tau)+H_1\tau\sigma+\sigma H_1\tau) x)=\partial(\sigma x)+a\partial(\tau x)+ b\partial((tI+H_1)\tau x).
\]
Comparing the coefficients of $1$ and $a$, we have $\sigma^2=\sigma$ and $\tau \sigma+\sigma\tau=\tau$. Since $u=t$, equation~\eqref{E:eqn8} yields $\sigma H_1=H_1\sigma$. Thus $\sigma$
is either zero or the identity, and, in either cases, $\tau=0$. Equation~\eqref{E:eqn5} shows that $\varphi$ is either zero or the identity; i.e., $\varphi$ is a trivial idempotent. Hence $V_t\into W$
is an indecomposable $(k\into D)$-module with $W=D_1^{(r_1)}\times\cdots\times D_s^{(r_s)}$.

For the rest of the proof, assume $\dim_k D_1\ge 4$. By Remark~\ref{R:remark1}, we need to consider the following four cases.

\begin{description}
\item [\textbf{Case (1)}] $\{1,a_1,a_1^2,b_1\}$ is linearly independent over $k$.\\
In equation~\eqref{E:eqn7}, we use descending induction to show that $H_1^j\tau H_1^i=0$ for all $i$, $j\ge 0$ (cf. the proof of Theorem 1.4 in~\cite{KrW} or the proof of Lemma 2.2 in~\cite{rW01}). Note that $H_1^j\tau H_1^i=0$ for $i\ge r_1$ and $j\ge r_1$. Thus we may assume inductively that $H_1^{j+1}\tau H_1^i=0$ and $H_1^j\tau H_1^{i+1}=0$. Set $H=(H_1,\ldots, H_s)\in \End_D(W)$, and observe that $H\partial (y)=\partial(H_1 y)$ for all $y\in k^{(r_1)}$. Replace $y$ by $H_1^i y$ in~\eqref{E:eqn7}, apply $H^j$ to both sides of the resulting equation and use the inductive hypothesis to obtain
\begin{align*}
-\partial(H_1^j\mu H_1^i y)+a\partial(H_1^j(\sigma-\rho)H_1^iy)+b\partial(H_1^j(t\sigma
-u\rho&+\sigma H_1 - H_1\rho)H_1^iy)\\&+ (a+tb)(a+ub)\partial(H_1^j\tau H_1^i y)=0.
\end{align*}
If $\{1,a_1,b_1,(a_1+tb_1)(a_1+ub_1)\}$ is linearly independent over $k$, then we conclude that $H_1^j\tau H_1^i=0$ for all $i$, $j\ge 0$. Setting $i=j=0$, we obtain $\tau=0$.

Thus equation~\eqref{E:eqn7} simplifies to equation~\eqref{E:eqn21}, and by linear independence of $\{1,a_1,b_1\}$ over $k$, we obtain equation~\eqref{E:eqn8}. As $\tau=0$, from equation~\eqref{E:eqn5} we have $\varphi(\partial(x))=\partial(\sigma x)$; thus, if
$\varphi$ is an isomorphism, then $\sigma$ is invertible, which contradicts equation~\eqref{E:eqn8} if $u\ne t$. Hence, if $u\ne t$ and $\{1,a_1,b_1,(a_1+tb_1)(a_1+ub_1)\}$ is linearly independent over $k$, then $(V_t\into W)\not \cong (V_u\into W)$.

To show that $V_t\into W$ is indecomposable, take $u=t$ and assume $\varphi$ is an idempotent. In equation~\eqref{E:eqn7}, use descending induction and the fact that $\{1,a_1,b_1,(a_1+tb_1)^2\}$ is linearly independent over $k$ to show that $\tau=0$. We obtain $\sigma^2=\sigma$ from equation~\eqref{E:eqn11}, and $\sigma H_1=H_1\sigma$ from equation~\eqref{E:eqn8}. Thus $\sigma$ is either zero or the identity, and we conclude that $\varphi$ is a trivial idempotent. Hence, if $\{1,a_1,b_1,(a_1+tb_1)^2\}$ is linearly independent over $k$, then $V_t\into W$ is an indecomposable $(k\into D)$-module with $W=D_1^{(r_1)}\times\cdots\times D_s^{(r_s)}$.

\item [\textbf{Case (2b)}] $\{1,a_1,a_1b_1,b_1\}$ is linearly
independent over $k$ and $a_1^2=b_1^2=0$.\\ Equation~\eqref{E:eqn7} simplifies to
\begin{equation}\label{E:eqn}
-\partial(\mu y)+a\partial((\sigma-\rho)y)+b\partial((t\sigma - u\rho+\sigma H_1 - H_1\rho)y)+ab\partial(((u+t) \tau + H_1\tau+\tau H_1)y)=0.
\end{equation}
Assuming $u\ne -t$, use descending induction in the equation above and the fact that $\{1,a_1,a_1b_1,b_1\}$ is linearly independent over $k$ to show that $\tau=0$ (in particular, observe that $\tau =0$ if $u=t$ and $\chr k\ne2$). Proceed as in the previous case to conclude that if $u\ne \pm t$, then $(V_t\into W)\not \cong (V_u\into W)$.

To show that $V_t\into W$ is indecomposable, take $u=t$ and assume $\varphi$ is an idempotent. If $\chr k\ne 2$, then $\tau=0$ as previously observed. If $\chr k=2$, then from equation~\eqref{E:eqn} we see that $\sigma H_1=H_1\sigma$ and $\tau H_1=H_1\tau$. Thus $\sigma$, $\tau\in k[H_1]$, and in particular $\sigma\tau=\tau\sigma$. Comparing the coefficients of $a$ in equation~\eqref{E:eqn11}, we conclude that $\tau=\sigma\tau+\tau\sigma=0$. In either case, $\tau =0$, and from equations~\eqref{E:eqn11} and~\eqref{E:eqn8}, we conclude that $\varphi$ is a trivial idempotent. Hence $V_t\into W$ is an indecomposable $(k\into D)$-module with $W=D_1^{(r_1)}\times\cdots\times D_s^{(r_s)}$.

\item [\textbf{Case (2c)}] The characteristic of $k$ is $2$, $(D_1,\m_1,K_1)$ is a local ring, $K_1$ is a purely inseparable extension of degree at least four and $a^2\in k$ for all $a\in K_1$.

Let $I=\m_1\times (0)\times \cdots \times (0)$, and pass to the Artinian pair $k\into D/I$, where $D/I=K_1\times D_2\times\cdots\times D_s$. By Lemma~\ref{L:lemma2}, it is enough to construct $|k|$ pairwise non-isomorphic indecomposable $(k\into D/I)$-modules $V\into W$, where $W=K_1^{(r_1)}\times D_2^{(r_2)}\times \cdots\times D_s^{(r_s)}$.

Choose $a_1$, $b_1\in K_1$ such that $\{1,a_1,a_1b_1,b_1\}$ is linearly independent over $k$, and construct a $(k\into D/I)$-module $V_t\into W$, where $W=K_1^{(r_1)}\times D_2^{(r_2)}\times \cdots\times D_s^{(r_s)}$, as in Construction~\ref{construction1}. Using linear independence of $\{1,a_1,a_1b_1,b_1\}$ over $k$ and the fact that $a_1^2$, $b_1^2\in k$, compare the coefficients of $a$, $b$ and $ab$ in equation~\eqref{E:eqn7} to obtain:
\begin{equation}\label{E:eqn9}
\sigma=\rho,\qquad  \sigma[(t-u)I+H_1]=H_1\sigma,  \qquad (t+u)\tau + H_1\tau+\tau H_1=0.
\end{equation}
If $u\ne t$, then $u+t\ne 0$ ($\chr k=2$), and, by descending induction, from the third equation in~\eqref{E:eqn9}, it follows that $\tau=0$. From equation~\eqref{E:eqn5}, we have $\varphi(\partial(x))=\partial(\sigma x)$; thus, if $\varphi$ is an isomorphism, then $\sigma$ is invertible, which contradicts the second equation in~\eqref{E:eqn9} if $u\ne t$. Hence, if $u\ne t$, then $(V_t\into W)\not\cong (V_u\into W)$.

To show $V_t\into W$ is indecomposable, take $u=t$ and assume $\varphi$ is an idempotent. From the second and third equations in~\eqref{E:eqn9}, it follows that $\sigma$, $\tau\in k[H_1]$; in particular, $\sigma\tau=\tau\sigma$. Comparing the coefficients of $1$ and $a$ in equation~\eqref{E:eqn11}, we see that $\sigma^2=\sigma$ and $\tau=\tau\sigma+\sigma\tau$. It follows that $\tau=0$. As $\sigma\in k[H_1]$, we conclude that $\sigma$ is either zero or the identity, and so is $\varphi$. Thus $\varphi$ is a trivial idempotent. Hence $V_t\into W$ is an indecomposable $(k\into
D/I)$-module with $W=K_1^{(r_1)}\times D_2^{(r_2)}\times \cdots\times D_s^{(r_s)}$.

\item [\textbf{Case (2d)}] The field $k$ has cardinality $2$ and $D_1=D_{1,1}\times\cdots\times D_{1,l}$, where $l\ge 4$ and each $(D_{1,i},\m_{1,i})$ is an Artinian local ring such that $D_{1,i}/\m_{1,i}=k$.

Let $W=(k^{(r_1)}\times \cdots\times k^{(r_1)})\times D_2^{(r_2)}\times \cdots\times D_s^{(r_s)}$, and let $V\into W$ be the $(k\into E)$-module given in Construction~\ref{construction2}. By Lemma~\ref{L:lemma2}, it is enough to show that $V\into W$ is indecomposable.

Let $\varphi\in \End_E(V,W)$ be an idempotent endomorphism. Write $\varphi=(\varphi_{1,1},\ldots,\varphi_{1,l},\varphi_2,\ldots,\varphi_s)$, where each $\varphi_{1,j}$, $j\in \{1,\ldots, l\}$, is an $r_1\times r_1$ matrix with entries in $k$ and each $\varphi_i$, $i\in \{2,\ldots, s\}$, is an $r_i\times r_i$ matrix with entries in $D_i$. Since $\varphi(V)\subseteq V$, there exist $r_1\times r_1$ matrices $\alpha$, $\beta$, $\gamma$ and $\delta$ with entries in $k$ such that
\begin{align}
\varphi(x,0,x,\ldots,x,\partial_2(x),&\ldots,\partial_s(x))\nonumber \\&=(\alpha x,\beta x, (\alpha+\beta) x,(\alpha+H_1\beta)x, \alpha x,\ldots,\alpha x,\partial_2(\alpha x),\ldots,\partial_s(\alpha
x))\label{E:eqn15}
\end{align}
and
\begin{equation*}
\varphi(0,y,y,H_1y,0,\ldots,0)=(\gamma y, \delta y, (\gamma+\delta)y,(\gamma+H_1\delta)y,\gamma y,\ldots,\gamma y,\partial_2(\gamma y), \ldots, \partial_s(\gamma y))
\end{equation*}
for all $x$, $y\in k^{(r_1)}$. Comparing components, we obtain equations which lead to the conditions:
\begin{enumerate}
\item [(i)] $\varphi_{1,j}=\alpha$ for all $j\in \{1,\ldots, l\}$, and
\item [(ii)] $\alpha H_1=H_1\alpha$.
\end{enumerate}
Since $\varphi$ is an idempotent, so is $\alpha$. As $\alpha H_1=H_1\alpha$, we conclude that $\alpha$ is either zero or the identity. From equation~\eqref{E:eqn15}, we see that if $\alpha$ is zero (respectively, the identity), then $\varphi_i$ is zero (respectively, the identity) for all $i\in\{2,\ldots,s\}$. Thus $\varphi$ is a trivial idempotent. Hence $V\into W$ is an indecomposable $(k\into E)$-module with $W=(k^{(r_1)}\times \cdots\times k^{(r_1)})\times D_2^{(r_2)}\times \cdots\times D_s^{(r_s)}$. \qedhere
\end{description}
\end{proof}

From the main theorem, we can deduce the structure of the monoid $\mcC(R)$ of isomorphism classes of maximal Cohen-Macaulay modules (together with $[0]$) when $\hat{R}/Q$ has infinite Cohen-Macaulay type for all minimal prime ideals $Q$ of $\hat{R}$. The following corollary is a special case of Theorem~3.15 in~\cite{BSunp}. Recall that the \textit{splitting number} $q$ of $R$ is defined by $q=|\minSpec(\hat{R})|-|\minSpec(R)|$, where $\minSpec(R)$ and $\minSpec(\hat{R})$ denote the set of minimal prime ideals of $R$ and $\hat{R}$ respectively.

\begin{corollary}
Let $(R,\m, k)$ be a one-dimensional analytically unramified local ring with splitting number $q$. Let $\hat{R}$ denote the $\m$-adic completion of $R$, and let $\Lambda$ denote the set of isomorphism classes of indecomposable maximal Cohen-Macaulay $\hat{R}$-modules. Assume $\hat{R}/Q$ has infinite Cohen-Macaulay type for all minimal prime ideals $Q$ of $\hat{R}$. 
\begin{enumerate}
\item If $q=0$, then $\mcC(R)\cong \mcC(\hat{R})\cong \N^{(\Lambda)}$.
\item If $q\ge 1$, then $\mcC(R)\cong \ker(\mcA(R))\cap \N^{(\Lambda)}$, where $\mcA(R)$ is a $q\times |\Lambda|$ integer matrix such that every element of $\Z^{(q)}$ occurs $|k|\cdot |\N|$
times as a column of $\mcA(R)$.
\end{enumerate}
\end{corollary}

\begin{proof}
If $q=0$, then the natural embedding $\mcC(R)\into \mcC(\hat{R})$ is surjective, and hence an isomorphism. 

Assume $q\ge 1$. Let $P_1$, \ldots, $P_s$ be the minimal prime ideals of $R$, and let $Q_{i,1}$,\ldots, $Q_{i,t_i}$ be the minimal prime ideals of $\hat{R}$ lying over the minimal prime ideal $P_i$ of $R$. Let $l\in \{0,\ldots, s-1\}$ denote the number of minimal prime ideals of $R$ with $t_i=1$. After renumbering (if necessary), assume that $P_1$, \ldots, $P_l$ are the minimal prime ideals of $R$ with $t_i=1$, and $P_{l+1}$, \ldots, $P_s$ are the minimal prime ideals of $R$ with $t_i\ge 2$.

By a corollary to a theorem due to Levy-Odenthal~\cite[Theorem~6.2]{LO96}, there exists a $q\times |\Lambda|$ integer matrix $\mcA(R)$ such that $\mcC(R)\cong \ker(\mcA(R))\cap \N^{(\Lambda)}$. The columns of the matrix $\mcA(R)$ are indexed by the isomorphism classes of indecomposable maximal Cohen-Macaulay $\hat{R}$-modules; if $M$ is an indecomposable maximal Cohen-Macaulay $\hat{R}$-module of rank $(r_{1,1},\ldots, r_{1,t_1},\ldots, r_{s,1},\ldots, r_{s,t_s})$, then the column of $\mcA(R)$ indexed by $[M]$ is 
\[
\begin{bmatrix} r_{l+1,1}-r_{l+1,2} & \cdots & r_{l+1,1}-r_{l+1,t_{l+1}} & \cdots &  r_{s,1}-r_{s,2} & \cdots & r_{s,1}-r_{s,t_s}
\end{bmatrix}^T.
\]

Let $\boldsymbol{a}=\begin{bmatrix} a_{l+1,2} & \cdots & a_{l+1,t_{l+1}}& \cdots & a_{s,2} & \cdots & a_{s,t_s}\end{bmatrix}^T\in \Z^{(q)}$. For each $i\in \{l+1, \ldots, s\}$, choose an integer~$b_i$ such that $b_i>\max\{|a_{i,j}| : 2\le j\le t_i\}$. Set 
\[
r_{i,1}=\begin{cases}
1 & \textrm{if $i\in \{1,\ldots, l\}$,}\\ b_i & \textrm{if $i\in \{l+1,\ldots, s\}$,} 
\end{cases}
\]
and for $i\in \{l+1,\ldots, s\}$ and $j\in \{2,\ldots, t_i\}$, set $r_{i,j}=r_{i,1}-a_{i,j}$. Since $\hat{R}/Q$ has infinite Cohen-Macaulay type for all minimal prime ideals $Q$ of $\hat{R}$, the main theorem guarantees the existence of an indecomposable maximal Cohen-Macaulay $\hat{R}$-module $M_{\boldsymbol{a}}$ of rank $\underline{r}=(r_{1,1},\ldots, r_{1,t_1},\ldots, r_{s,1},\ldots, r_{s,t_s})$. By construction, the column indexed by $[M_{\boldsymbol{a}}]$ is $\boldsymbol{a}$.

Observe that the main theorem also guarantees the existence of an indecomposable maximal Cohen-Macaulay $\hat{R}$-module $M_{\boldsymbol{a},n}$ of rank $\underline{r}+(n,\ldots, n)$, $n\in \N$. By construction, the column indexed by $[M_{\boldsymbol{a}, n}]$ is $\boldsymbol{a}$, and we conclude that $\boldsymbol{a}$ appears $|\N|$ times as a column of $\mcA(R)$. If $k$ is infinite, then the main theorem guarantees the existence of $|k|$ pairwise non-isomorphic indecomposable maximal Cohen-Macaulay $\hat{R}$-modules of the same rank $\underline{r}+(n,\ldots, n)$; note that the columns indexed by these indecomposable modules are equal to $\boldsymbol{a}$. Hence $\boldsymbol{a}$ appears $|k|\cdot |\N|$ times as a column of $\mcA(R)$.
\end{proof}

\section{Ranks of indecomposable maximal Cohen-Macaulay modules over $R=k[\![x,y]\!]/((x^3-y^7)x)$} \label{S:example}

The main theorem describes which tuples can be realized as ranks of indecomposable maximal Cohen-Macaulay $R$-modules when $R/P_1$ has infinite Cohen-Macaulay type. In this case, for every non-trivial $s$-tuple $(r_1, \ldots, r_s)$ of non-negative integers, there exists an indecomposable maximal Cohen-Macaulay $R$-module of rank $(r_1,\ldots, r_s)$ provided $r_1\ge r_i$ for all $i\in \{1,\ldots, s\}$. What happens if $r_1<r_i$ for some $i$?

We give an example of a complete local ring $R$ of infinite Cohen-Macaulay type with exactly two minimal prime ideals $P_1$ and $P_2$ such that $R/P_1$ has infinite Cohen-Macaulay type (and
$R/P_2$ has finite Cohen-Macaulay type), for which there are no indecomposable maximal Cohen-Macaulay $R$-modules of rank $(1,n)$ for $n\ge 4$.

Consider the hypersurface
\[
R=k[\![x,y]\!]/((x^3-y^7)x)
\]
with minimal prime ideals $P_1=(x^3-y^7)/((x^3-y^7)x)$ and $P_2=(x)/((x^3-y^7)x)$. Observe that $R/P_1\cong k[\![x,y]\!]/(x^3-y^7)\cong k[\![t^7,t^3]\!]$ has infinite Cohen-Macaulay type, and $R/P_2\cong k[\![y]\!]$ is a discrete valuation ring, hence of finite Cohen-Macaulay type. %Note that a maximal Cohen-Macaulay $R$-module
%of rank $(0,n)$ is an $R/P_2$-module. As the rank of a module over a
%discrete valuation ring is at most one, there do not exist
%indecomposable maximal Cohen-Macaulay $R$-modules of rank $(0,n)$ if
%$n\ge 2$.

The integral closure $\overline{R}$ of $R$ is $\overline{R}=k[\![t]\!]\times k[\![y]\!]$, and the map
$\varphi\colon R\into \overline{R}$ is given by:
\begin{equation}\label{E:e02}
\bar{x} \mapsto (t^7,0),\quad \bar{y} \mapsto (t^3,y).
\end{equation}
Note that $\overline{R}$ is generated as an $R$-module by $(0,1)$ and $(t^i,0)$ for $i\in \{0,1,2\}$. The first Drozd-Ro{\u\i}ter condition fails, and so $R$ has infinite Cohen-Macaulay type.

The conductor of $R$ is $\cond=(x^3,xy^4,y^7)/((x^3-y^7)x)$, and the conductor square is:
\[
\xymatrix{\ds R=\frac{k[\![x,y]\!]}{((x^3-y^7)x)}\ar[r]\ar[d] & k[\![t]\!]\times k[\![y]\!]=\overline{R}\ar[d]\\
\ds A=\frac{R}{\cond}=\frac{k[\![x,y]\!]}{(x^3,xy^4,y^7)} \ar[r] &
\ds \frac{k[\![t]\!]}{(t^{19})}\times \ds\frac{k[\![y]\!]}{(y^7)}
=\frac{\overline{R}}{\cond}=B, }
\]
where $R_{\normalfont\text{art}}=(A\into B)$. Write $B=B_1\times B_2$, where $B_1=k[\![t]\!]/(t^{19})$ and $B_2=k[\![y]\!]/(y^7)$. Reducing the Artinian pair $A\into B$ modulo the maximal ideal $\n$ of $A$, we obtain the Artinian pair $k\into D$, where $D=D_1\times D_2=k[\![t]\!]/(t^3)\times k$.

\subsection{Matrix representation}
Given an $(A\into B)$-module $V\into W$, where $W=B_1^{(r_1)}\times B_2^{(r_2)}$, one can represent $V\into W$ by matrices in the following way. An element $w$ of $W$ can be written as a column vector
\[
w=\left[\begin{array}{c} w_1\\ \hline w_2\end{array}\right],
\]
where each $w_i$ is a $r_i\times 1$ column vector with entries in $B_i$. Let $\{v_1, \ldots, v_m\}$ be a minimal set of generators of $V$ as an $A$-module. Since $V$ is an $A$-submodule of $W$, we can write each generator $v_j$ as
\[
v_j=\left[\begin{array}{c} v_{1,j}\\ \hline
v_{2,j}\end{array}\right].
\]
Thus $V$ is the $A$-column span of the matrix:
\[
Q=\left[\begin{array}{cccc}v_{1,1} & v_{1,2} & \cdots & v_{1,m} \\
\hline v_{2,1} & v_{2,2} & \cdots & v_{2,m} \end{array}\right].
\]
Note that each row is a block matrix of size $r_i\times m$ as each $v_{i,j}$ is an $r_i\times 1$ column vector.

As done in~\cite{GR78}, \cite{rWsW94} and~\cite{C94}, we can perform row and column operations without changing the module. In particular, observe the following.
\begin{enumerate}
\item [(a)] Column operations over $A$ correspond to multiplication on the right by an invertible matrix $R(x,y)$ with entries in $A$ (see the map $\varphi$ in~\eqref{E:e02}):
\[ \left[
\begin{array}{c} Q_1\\ \hline Q_2\end{array}\right] R(x,y)=\left[
\begin{array}{c} Q_1R(t^7,t^3)\\ \hline Q_2R(0,y)\end{array}\right].
\]
Note that column operations over $A$ do not change the module.

\item [(b)] Row operations over $B$ correspond to multiplication on the left by a matrix $\ds \left[\begin{array}{c} T_1\\ \hline T_2\end{array}\right]$, where $T_1$ is an $r_1\times r_1$ invertible matrix with entries in $B_1$ and $T_2$ is an $r_2\times r_2$ invertible matrix with entries in $B_2$:
\[
\left[\begin{array}{c} T_1\\
\hline T_2\end{array}\right]\left[ \begin{array}{c} Q_1\\ \hline
Q_2\end{array}\right]=\left[ \begin{array}{c} T_1Q_1\\ \hline
T_2Q_2\end{array}\right].
\]
Note that row operations over $B$ give a module isomorphic to the original one.
\end{enumerate}

Looking at the matrix representation of $V\into W$, it is possible to determine whether the module is indecomposable or not. In particular, we use the following remark to determine whether an
$(A\into B)$-module $V\into W$, where $W=B_1\times B_2^{(n)}$, decomposes.

\begin{remark}\label{R:decomp}
Consider an $(A\into B)$-module $V\into W$, where $W=B_1\times B_2^{(n)}$. Assume $V$ is generated minimally by $m$ elements and is given by the $A$-column space of the matrix
\[
Q=\left[\begin{array}{ccc}  u_1 & \cdots & u_m \\
\hline & Q_2 &
\end{array}\right],
\]
where $u_i\in B_1$ and $Q_2$ is an $n\times m$ matrix over $B_2$.

If, after row and column operations, the matrix $Q$ is equivalent to the matrix
\[
Q'=\left[\begin{array}{cc} 1 \quad u_2 \quad \cdots \quad u_{n-1}
\quad 0 & u_{n+1} \quad \cdots \quad u_m\\ \hline I_n &
\textbf{0}_{n\times (m-n)}
\end{array}\right],
\]
then $V\into W$ decomposes as a direct sum of two modules; that is,
\[
(V\into W)\cong (V_1\into W_1)\oplus (V_2\into W_2),
\]
where $W_1=B_1\times B_2^{(n-1)}$, $W_2=\{0\}\times B_2$, and $V_1$ and $V_2$ are the $A$-column spaces of the matrices
\[
U_1=\left[\begin{array}{cc} 1 \quad u_2 \quad \cdots \quad u_{n-1} &
u_{n+1} \quad \cdots \quad u_m\\ \hline I_{n-1} &
\textbf{0}_{(n-1)\times
(m-n)} \end{array}\right] \qquad \text{ and } \qquad U_2=\left[\begin{array}{c} \star\\
\hline 1 \end{array}\right].
\]
(Here $\star$ denotes the empty matrix.)
\end{remark}

\subsection{Maximal Cohen-Macaulay $R$-modules of rank $(1,n)$}

We know that there exist an indecomposable maximal Cohen-Macaulay $R$-module of rank $(1,0)$ (namely, the $R$-module $R/P_1$), and an indecomposable maximal Cohen-Macaulay $R$-module of constant rank $(1,1)$. Moreover, there exist indecomposable maximal Cohen-Macaulay $R$-modules of rank $(1,2)$ and $(1,3)$. In fact, by Proposition~\ref{P:local} and Lemma~\ref{L:lemma2}, it is enough to construct indecomposable $(k\into D)$-modules $V_1\into W_1$ and $V_2\into W_2$, where $W_1=D_1\times D_2^{(2)}$ and $W_2=D_1\times D_2^{(3)}$. If $V_1$ and $V_2$ are the $A$-column spaces of the matrices
\[
Q_1=\left[\begin{array}{cc}1 &
t\\ \hline 1 & 0\\ 0 & 1\end{array}\right] \quad \textrm{ and } \quad Q_2=\left[\begin{array}{ccc}1 & t & t^2\\ \hline 1 & 0 & 0\\
0 & 1 & 0\\ 0 & 0 & 1
\end{array}\right],
\]
then $V_1 \into W_1$ and $V_2 \into W_2$ are indecomposable $(k\into D)$-modules with $W_1=D_1\times D_2^{(2)}$ and $W_2=D_1\times D_2^{(3)}$.

We show that there do \textit{not} exist indecomposable maximal Cohen-Macaulay $R$-modules of rank $(1,n)$ for $n\ge 4$. Note that $\overline{R}$ is a direct product of local rings; thus, by Theorem~\ref{T:theorem2}, it is enough to prove the following.

\begin{proposition}
Let $R=k[\![x,y]\!]/((x^3-y^7)x)$, and let $P_1=(x^3-y^7)/((x^3-y^7)x)$ and $P_2=(x)/((x^3-y^7)x)$ be the minimal prime ideals of $R$. Let $R_{\normalfont\text{art}}=(A\into B)$ be the Artinian pair associated to $R$, where $B=B_1\times B_2=k[\![t]\!]/(t^{19})\times k[\![y]\!]/(y^7)$. If $n\ge 4$, then there do not exist indecomposable $R_{\normalfont\text{art}}$-modules $V\into W$ with $W=B_1\times B_2^{(n)}$.
\end{proposition}

\begin{proof}
Fix $n\ge 4$. Let $V\into W$ be an $(A \into B)$-module, where $W=B_1\times B_2^{(n)}$. Assume $V$ is generated minimally by $m$ elements and is the $A$-column space of
\[
Q=\left[ \begin{array}{c} Q_1\\ \hline Q_2\end{array}\right] =
\left[
\begin{array}{ccc} u_1& \cdots & u_m\\ \hline v_{1,1} & \cdots &
v_{1,m}\\ \vdots & & \vdots \\ v_{n,1} & \cdots & v_{n,m}
\end{array}\right],\qquad u_i\in
B_1=k[\![t]\!]/(t^{19}), \quad v_{i,j}\in
B_2=k[\![y]\!]/(y^7).
\]
Since $V\into W$ is an $(A\into B)$-module, we have $BV=W$. This relation means that
\begin{enumerate}
\item [(i)] the $B_1$-span of the columns of $Q_1$ equals $B_1$, and
\item [(ii)] the $B_2$-span of the columns of $Q_2$
equals $B_2^{(n)}$.
\end{enumerate}
In particular, (i) implies that the $1\times m$ matrix $Q_1=(u_i)$ has a unit element. Similarly, (ii) implies that the $n\times m$ matrix $Q_2=(v_{i,j})$ has an $n\times n$ invertible sub-matrix. As a consequence, $m\ge n$.

\bigskip

\noindent \textbf{Step 1.}\\
Since $Q_2$ has an $n\times n$ invertible sub-matrix, first we use row operations over $B$ (and eventually swap columns) to obtain the identity matrix $I_n$ on the bottom part of $Q$, and then we use column operations over $A$ to clear every element in each row of $Q_2$ except the entry $(i,i)$. Observe that this last step can be done because $R/P_2\cong k[\![y]\!]$ is a discrete valuation ring, and so the map $A\to B_2=k[\![y]\!]/(y^7)$ is surjective. Now
\begin{equation}\label{E:e1}
Q_2=\begin{bmatrix} I_n & \textbf{0}_{n\times (m-n)}\end{bmatrix}.
\end{equation}
Let $e_1$ be the unit element in $Q_1$. If $e_1$ occurs in the first $n$ columns of $Q_1$, then we can permute columns so that $e_1$ appears in the first column of $Q_1$. By permuting the rows of
$Q_2$, we can recover the form~\eqref{E:e1}. On the other hand, if $e_1$ occurs in the last $m-n$ columns, then we can do one or more column operations without affecting $Q_2$ to put $e_1$ in the first column of $Q_1$.

Thus we can assume that $V$ is (isomorphic to) the $A$-column space of
\begin{equation}\label{E:e2}
Q=\left[ \begin{array}{c} Q_1 \\
\hline Q_2 \end{array}\right]=\left[ \begin{array}{cc} 1\quad u_2
\quad \cdots & \cdots \quad u_m\\ \hline I_n &
\textbf{0}_{n\times(m-n)}
\end{array} \right].
\end{equation}

\bigskip

\noindent \textbf{Step 2.}\\ In this step, we do column operations over $A$ involving the first $n$ columns in order to clean up $Q_1$. After each step, we restore $Q_2$ to the form~\eqref{E:e1} via row operations over $B$.

Write each element $u_i\in B_1$, $i\in \{2,\ldots,m\}$, as $u_i=\sum_{j=0}^{18} a_{i,j}t^j$ for some $a_{i,j}\in k$. Since $u_1=1$, we can use column operations over $A$ to clear the terms involving $t^{r}$ in $u_2$, \ldots, $u_n$ for $r$ in the additive semigroup generated by $0$, $3$ and $7$. Thus the first $n$ elements in $Q_1$ are of the form:
\[
u_1=1, \qquad
u_i=a_{i,1}t+a_{i,2}t^2+a_{i,4}t^4+a_{i,5}t^5+a_{i,8}t^8+a_{i,11}t^{11},
\quad a_{i,j}\in k, \quad 2\le i\le n.
\]

If there is $i\in \{2,\ldots, n\}$ such that $a_{i,1}\ne 0$ (without loss of generality, assume $i=2$), then we can use column operations over $A$ to clear the terms involving $t$, $t^4$, $t^8$ and $t^{11}$ in $u_3$, \ldots, $u_n$. Otherwise, we have $a_{i, 1}=0$ for all $i\in \{1,\ldots, n\}$. Thus either
\begin{equation}\label{E:e110}
u_2=a_{2,1}t+a_{2,2}t^2+a_{2,4}t^4+a_{2,5}t^5+a_{2,8}t^8+a_{2,11}t^{11},
\qquad u_i=a_{i,2}t^2+a_{i,5}t^5, \quad 3\le i\le n,
\end{equation}
or
\begin{equation}\label{E:e120}
u_i=a_{i,2}t^2+a_{i,4}t^4+a_{i,5}t^5+a_{i,8}t^8+a_{i,11}t^{11},
\quad 2\le i\le n.
\end{equation}

In case \eqref{E:e110}, if there is $i\in \{3,\ldots, n\}$ such that $a_{i,2}\ne 0$ (without loss of generality, assume $i=3$), then we can use column operations over $A$ to clear the terms involving $t^2$, $t^5$, $t^8$ and $t^{11}$ in $u_2$, $u_4$, \ldots, $u_n$, and to obtain $u_i=0$ for all $i\in \{4, \ldots, n\}$. If, on the other hand, $a_{i,2}=0$ for all $i\in \{3, \ldots, n\}$, then
\[
u_2=a_{2,1}t+a_{2,2}t^2+a_{2,4}t^4+a_{2,5}t^5+a_{2,8}t^8+a_{2,11}t^{11},
\qquad u_i=a_{i,5}t^5, \quad 3\le i\le n,
\]
and we can use column operations over $A$ to obtain $u_i=0$ for all $i\in \{4,\ldots, n\}$. 

In case \eqref{E:e120}, if there is $i\in \{2,\ldots, n\}$ such that $a_{i,2}\ne 0$ (without loss of generality, assume $i=2$), then we can use column operations over $A$ to clear the terms involving $t^2$, $t^5$, $t^8$ and $t^{11}$ in $u_3$, \ldots, $u_n$. Thus 
\[
u_2=a_{2,2}t^2+a_{2,4}t^4+a_{2,5}t^5+a_{2,8}t^8+a_{2,11}t^{11}, \qquad
u_i=a_{i,4}t^4, \quad 3\le i\le n,
\]
and we can use column operations over $A$ to obtain $u_i=0$ for all $i\in \{4,\ldots, n\}$. If, on the other hand, $a_{i,2}=0$ for all $i\in \{1, \ldots, n\}$, then 
\[
u_i=a_{i,4}t^4+a_{i,5}t^5+a_{i,8}t^8+a_{i,11}t^{11}, \quad 2\le i\le
n.
\]
If there is $i\in \{2,\ldots, n\}$ such that $a_{i,4}\ne 0$ (without loss of generality, assume $i=2$), then we can use column operations over $A$ to clear the terms involving $t^4$ and $t^{11}$ in $u_3$, \ldots, $u_n$. Otherwise, we have $a_{i,4}=0$ for all $i\in \{1, \ldots, n\}$. Thus either
\begin{equation}\label{E:e190}
u_2=a_{2,4}t^4+a_{2,5}t^5+a_{2,8}t^8+a_{2,11}t^{11}, \qquad
u_i=a_{i,5}t^5+a_{i,8}t^8, \quad 3\le i\le n,
\end{equation}
or
\begin{equation}\label{E:e200}
u_i=a_{i,5}t^5+a_{i,8}t^8+a_{i,11}t^{11}, \quad 2\le i\le n.
\end{equation}

In case \eqref{E:e190}, if there is $i\in \{3,\ldots, n\}$ such that $a_{i,5}\ne 0$ (without loss of generality, assume $i=3$), then we can use column operations over $A$ to clear the terms involving $t^5$, $t^8$ and $t^{11}$ in $u_2$, $u_4$, \ldots, $u_n$, and to obtain $u_i=0$ for all $i\in \{4,\ldots, n\}$. If, on the other hand, $a_{i,5}=0$ for all $i\in \{3, \ldots, n\}$, then 
\[
u_2=a_{2,4}t^4+a_{2,5}t^5+a_{2,8}t^8+a_{2,11}t^{11}, \qquad u_i=a_{i,8}t^8,
\quad 3\le i\le n,
\]
and we can use column operations over $A$ to obtain $u_i=0$ for all $i\in \{4,\ldots, n\}$.

In case \eqref{E:e200}, if there is $i\in \{2,\ldots, n\}$ such that $a_{i,5}\ne 0$ (without loss of generality, assume $i=2$), then we can use column operations over $A$ to clear the terms involving $t^5$, $t^8$ and $t^{11}$ in $u_3$, \ldots, $u_n$, and to obtain $u_i=0$ for all $i\in \{3,\ldots, n\}$. If, on the other hand, $a_{i,5}=0$ for all $i\in \{1, \ldots, n\}$, then 
\[
u_i=a_{i,8}t^8+a_{i,11}t^{11}, \quad 2\le i\le n.
\]
If there is $i\in \{2,\ldots, n\}$ such that $a_{i,8}\ne 0$ (without loss of generality, assume $i=2$), then we can use column operations over $A$ to clear the terms involving $t^8$ and $t^{11}$ in $u_3$, \ldots, $u_n$, and to obtain $u_i=0$ for all $i\in \{3,\ldots n\}$. If, on the other hand, $a_{i,8}=0$ for all $i\in \{1, \ldots, n\}$, then 
\[
u_i=a_{i,11}t^{11}, \quad 2\le i\le n,
\]
and we can use column operations over $A$ to obtain $u_i=0$ for all $i\in \{3,\ldots, n\}$.

%Hence
%\[
%\begin{cases}
%u_1=1,\\
%u_2=a_{2,1}t+a_{2,2}t^2+a_{2,4}t^4+a_{2,5}t^5+a_{2,8}t^8+a_{2,11}t^{11},\\
%u_3=a_{3,2}t^2+a_{3,4}t^4+a_{3,5}t^5+a_{3,8}t^8,\\ u_i=0, \quad 4\le i \le n.
%\end{cases}
%\]
In conclusion, for $n\ge4$, at least one element $u_i$ in $Q_1$ is zero for some $i\in \{1,\ldots, n\}$. By Remark~\ref{R:decomp}, the matrix $Q$ decomposes and so does the module $V\into W$.
\end{proof}

\section*{Acknowledgements}
The authors thank Roger Wiegand for his helpful insights and numerous suggestions, and careful reading of earlier versions of this article.

\bibliographystyle{amsplain}
\bibliography{referencesResearch}

\end{document}